\documentclass{article}

\PassOptionsToPackage{sort}{natbib}

\usepackage{natbib}

\usepackage[utf8]{inputenc} 
\usepackage[T1]{fontenc}    
\usepackage{hyperref}       
\usepackage{url}            
\usepackage{booktabs}       
\usepackage{caption}
\usepackage{amsfonts}       
\usepackage{nicefrac}       
\usepackage{microtype}      
\usepackage{xcolor}         
\usepackage{amsmath}
\usepackage{amssymb}
\usepackage{fullpage}
\usepackage{amsthm}
\usepackage{multirow}
\usepackage{makecell}
\usepackage{placeins}
\usepackage[capitalise]{cleveref}

\usepackage{shortcutsg}
\usepackage{bm}
\usepackage{tikz}
\usepackage{subcaption}
\Crefname{assumption}{Assumption}{Assumptions}
\usetikzlibrary{arrows.meta,positioning,arrows,decorations}
\usepackage{todonotes}

\theoremstyle{plain}
\newtheorem{theorem}{Theorem}
\newtheorem{corollary}{Corollary}
\newtheorem{lemma}[theorem]{Lemma}
\theoremstyle{definition}
\newtheorem{assumption}{Assumption}

\newtheorem{definition}{Definition}

\newtheorem{remark}{Remark}

\newcommand{\new}{\mathrm{new}}
\newcommand{\IPW}{\mathrm{IPW}}
\newcommand{\T}{\mathcal{T}}

\usepackage{xcolor}

\title{Nonparametric Jackknife Instrumental Variable Estimation and Confounding Robust Surrogate Indices}

\author{Aur\'elien Bibaut\thanks{Netflix} \and  Nathan
Kallus\thanks{Netflix, Cornell} \and Apoorva Lal\thanks{Netflix}}

\begin{document}
\maketitle

\begin{abstract}
Jackknife instrumental variable estimation (JIVE) is a classic method to leverage many weak instrumental variables (IVs) to estimate linear structural models, overcoming the bias of standard methods like two-stage least squares. 
In this paper, we extend the jackknife approach to nonparametric IV (NPIV) models with many weak IVs. Since NPIV characterizes the structural regression as having residuals projected onto the IV being zero, existing approaches minimize an estimate of the average squared projected residuals, but their estimates are biased under many weak IVs. We introduce an IV splitting device inspired by JIVE to remove this bias, and by carefully studying this split-IV empirical process we establish learning rates that depend on generic complexity measures of the nonparametric hypothesis class. 
We then turn to leveraging this for semiparametric inference on average treatment effects (ATEs) on unobserved long-term outcomes predicted from short-term surrogates, using historical experiments as IVs to learn this nonparametric predictive relationship even in the presence of confounding between short- and long-term observations. Using split-IV estimates of a debiasing nuisance, we develop asymptotically normal estimates for predicted ATEs, enabling inference.
\end{abstract}

\section{Introduction}

The non-parametric instrumental variable (NPIV) problem is, given $N$ observations, $O_i=(A_i,S_i,Y_i)$ $i=1,\dots,N$, on instrumental variables (IVs) $A$, interventions $S\in\RR^p$, and outcomes $Y\in\RR$, to find a function $h:\RR^p\to\RR$ satisfying
\begin{align}\label{eq: npiv}
\EE[Y-h(S)\mid A]=0.
\end{align}
This can be motivated by a structural (\ie, causal) model
$Y=h^\star(S)+\epsilon$, where $S$ and $\epsilon$ can be endogenous
due, \eg, to the presence of common confounders (so that
$h^\star(S)\neq \EE[Y\mid S]$), but $A$ and $\epsilon$ are exogenous
so that \cref{eq: npiv} holds for $h^\star$. A variety of work studies
the estimation of $h^\star$ and inference on functionals thereof under
nonparametric restrictions on $h^\star$ as we receive additional
observations $N$ from a fixed $O=(A,S,Y)$ distribution
\citep{newey2003instrumental,severini2006some,ai2003efficient,ai2007estimation,ai2012semiparametric,chen2009efficient,severini2012efficiency,bennett2023variational,bennett2019deep,bennett2023inference,chen2012estimation,darolles2011nonparametric,dikkala2020minimax,hartford2017deep,santos2011instrumental,singh2019kernel,zhang2023instrumental,kremer2022functional,kremer2023estimation,bennett2023source,bennett2023source}.

In this paper, we tackle NPIV in the challenging many-weak-IV setting, where $A\in\{1,\dots,K\}$ is discrete and we only see so many of each value, namely for each $a \in [K]$, we have $n$ i.i.d. observations of $(Y,S) \mid A=a$, forming $N=nK$ observations in total.

In the linear setting where $h^\star(S)=S\tr\beta^\star$, 
\cref{eq: npiv} reduces to solving $\EE[Y\mid A]=\EE[S\mid A]\tr\beta$ for $\beta$. This motivates the two-stage least squares (2SLS) approach of estimating $\beta$ by ordinary least squares (OLS) of $Y$ on the ``first-stage" OLS prediction of $S$ given $A$ (for discrete $A$ this is simply the sample means of $S$ for each $A$ value). However, when $n\ll N$, even as $N\to\infty$ this can incur non-vanishing bias because the first-stage regression may not converge at all \citep{angrist1999jackknife,peysakhovich2018learning,bibaut2024learning}.
JIVE \citep{angrist1999jackknife} addresses this by regressing $Y$ on a prediction of $S$ given $A$ based on OLS using all the data \textit{except} the datapoint on which we make the prediction. This renders the errors from the first stage uncorrelated with the second stage so they average out to zero so that we regain consistency \citep{chao2012asymptotic}.

The NPIV analog, which we tackle, is, however, unresolved. It is also rather nuanced because when $S$ is continuous but $A$ is discrete, a nonparametric $h^\star$ is generally not uniquely identified by \cref{eq: npiv}, which involves just $K$ moments but a general function $h$. Nonetheless, certain \textit{linear functionals} of $h^\star$, meaning $\theta_0=\EE[\alpha(S)h^\star(S)]$ for some $\alpha$, may still be uniquely identified, meaning $\theta_0=\EE[\alpha(S) h(S)]$ for \textit{any} $h$ satisfying \cref{eq: npiv}.

This problem setting is of particular interest in digital experimentation, where the rapid pace of innovation means we have many ($K$) historical randomized experiments (with serial numbers $A$), which can be used to instrument for the effect ($h^\star$) of short-term surrogate observations ($S$) on long-term outcomes ($Y$) even in the presence of unobserved confounding between the two, but where each experiment has a certain sample size ($n$). If we know this effect, we can construct a surrogate index $h^\star(S)$ such that average treatment effects (ATEs) on $Y$ are the same as those on $h^\star(S)$. Moreover, the ATEs on $h^\star(S)$ is a linear functional thereof.
Then, for novel experiments, predicting long-term ATEs before observing
long-term outcomes can be phrased as inference on a linear functional
of a solution to \cref{eq: npiv}. 


In this paper we also tackle the question how to reliably do this inference in the presence of underidentified and nonparameteric $h^\star$, which is another significant challenge, besides solving \cref{eq: npiv} in the many-weak-IV setting. We furthermore extend the simple instrumentation identification to account for the possibility that short-term surrogate observations do not fully mediate the treatment effects on long-term outcomes (that is, there is exclusion violation).

In this paper, we develop both a novel estimator for $h^\star$ in the nonparametric many-weak-IV setting and methods for debiased inference on surrogate-predicted ATEs.
The contributions and organization of the paper are as follows
\begin{enumerate}
\item In \cref{sec: npjive}, we propose a novel nonparametric jackknife IV estimator (npJIVE) based on estimating the average of squared moments using a split-IV device inspired by JIVE and then minimizing it over a generic nonparametric hypothesis class with Tikhonov regularization.
\item In \cref{sec: npjive analysis}, we prove that npJIVE is consistent to the minimum norm function satisfying \cref{eq: npiv} and that its average of squared moments converges at a rate governed by the functional complexity of the hypothesis. To our knowledge, these guarantees are the first of their kinds for NPIV in the many-weak-IV setting.
\item In \cref{sec: surrogates}, we study the above IV-surrogate-index approach to long-term causal inference and extend it to also tackle the setting where
$S$ does not fully mediate all of the effect on $Y$. We establish identification conditions where we can combine historical experiments with full long-term observations together with a new experiment with only short-term observations in order to identify the long-term causal effect in the new experiment, even in the presence of confounding and exclusion violations.
\item In \cref{sec: semiparam}, we a develop methods for inference on the new experiment's long-term causal effect. We devise a debiased estimator that involves a new debiasing nuisance related to the ratio of densities of short-term outcomes in the old and the new experiments. We develop an estimator for this debiasing nuisance with guarantees in the many-weak-IV setting. And, we show that when we combine npJIVE with this debiasing nuisance in the right way, we can obtain asymptotically normal estimates of the long-term causal effect in the new experiment, even when npJIVE and the debiasing nuisance are learned completely nonparametrically. This ensures fast estimation rates and implies simple Wald confidence intervals give asymptotically correct coverage.
\end{enumerate}

Taken together, our methods and results provide new ways to conduct long-term causal inference in challenging, but practically very relevant, settings.

\section{The npJIVE Estimator for NPIV with Many Weak IVs}\label{sec: npjive}

We now motivate and define npJIVE. 
For function $f(A,S,Y)$ let us write $[\T_K f](A)=\EE[f(A,S,Y)\mid A]$
and $\|f\|^2=\frac1K\sum_{a=1}^K\|f\|^2_{(a)}$ with $\|f\|^2_{(a)}=\EE[f^2(a,S,Y)\mid A=a]$, 
and let $L_2(S)=\{h(S):\|h\|<\infty\}$ be the square integrable functions of $S$. 

Let
$$
\Hcal_0=\{h\in L_2(S):R_1(h)=0\},~~\text{where}~~
R_1(h) = 
\frac12 \|\T_K (h - Y)\|^2,
$$
be the solutions to \cref{eq: npiv}.

Given a hypothesis class $\Hcal\subseteq L_2(S)$, suppose $\mathcal{H}$ contains a solution $h_0\in\Hcal\cap\Hcal_0\neq\varnothing$. Then $h_0\in\argmin_{h\in\Hcal}R_1(h)\subseteq \Hcal_0$, meaning the solution must minimize the risk criterion $R_1(h)$ over $\mathcal{H}$ and any such minimizer must solve \cref{eq: npiv}.

Then, an approach to recover such a solution from data would be to estimate $R_1$ and minimize it. Perhaps the most natural empirical analog of $R_1$ is the empirical ``plug-in'' risk
\begin{align}
    \widehat{R}_1^{\mathrm{plug-in}}(h) =
    \frac{1}{2 K} \sum_{a = 1}^K  ([\widehat \T_K (Y - h)](a))^2,~~\text{where}~~[\widehat \T_K f](a) =\frac1n \sum_{i: A_i = a} f(a, S_i, Y_i),
\end{align}
meaning $[\widehat \T_K f](a)$ is the sample average over the $A_i=a$ data.
In fact, optimizing $\widehat{R}_1^{\mathrm{plug-in}}(h)$ over linear functions $h(S)=S\tr\beta$ is \textit{almost} 2SLS (aside from some $A$-dependent sample weights). And, optimizing it over a nonparametric hypothesis class $\Hcal$ (possibly with regularization) is \textit{exactly} the adversarial NPIV estimators of \citep{bennett2019deep,bennett2023inference,bennett2023source,bennett2023variational,dikkala2020minimax} with an adversary function class that includes all functions of the discrete $A$.

Unfortunately, this is not an unbiased estimate of $R_1(h)$, and it can be inconsistent when $n\not\to\infty$. Essentially, given $n$ observations on some variables $(Z,Z')$, the expectation of the product of sample means of $Z_i$ and of $Z_i'$ is not the product of the expectations of $Z$ and $Z'$ -- there's also $1/n$ times the covariance of $Z$ and $Z'$.
This is the fundamental reason why 2SLS isn't consistent under many weak IVs. 
One way to see the key trick in JIVE is that, partitioning observations in folds $1,\ldots,L$ of equal size, with $V_i \in [L]$ the fold membership of observation $i$, an unbiased estimate of the product of expectations is the product of the empirical mean of $Z_i$ over folds $1, \ldots, L-1$ with the empirical mean of $Z_i'$ over fold $L$. (JIVE uses $L=n$, corresponding to ``leave one out" sample means).

Drawing inspiration from this insight, we propose a 2-fold cross-fold risk estimator:
\begin{align*}
    \widehat R_{1}(h) &= \frac{1}{2 K} \sum_{a = 1}^K  [\widehat \T_{K,0} (Y - h)](a) [\widehat \T_{K_1} (Y-h)](a),\\
    \text{where}~~[\widehat \T_{K,v} f](a) &= \frac{2}{n} \sum_
    {i: A_i = a, V_i = v} f(a,S_i,Y_i),
\end{align*}
assuming $n$ is even.
Meaning $[\widehat \T_{K,v} f](a)$ takes the sample average over the observations for which $A_i = a$ belonging to fold $v$.

Our proposed npJIVE estimator is to minimize this risk plus a Tikhonov regularization penalty
\begin{equation}\textstyle
\hat h=\argmin_{h\in\Hcal}~\widehat R_{1}(h)+\lambda \| h \|_{2,N}^2,\tag{npJIVE}
\end{equation}
where $\|h\|_{2,N}^2=\frac1N\sum_{i=1}^Nh^2(S_i)$. In fact, optimizing $\widehat{R}_1(h)$ over linear functions $h(S)=S\tr\beta$ is \textit{almost} JIVE (aside from leave-one-out instead of 2 folds and some $A$-dependent sample weights).

\subsection{Analysis of npJIVE}\label{sec: npjive analysis}

We now give an analysis of the convergence of npJIVE, both in terms of violations of \cref{eq: npiv} and in terms of distance to a minimum norm solution.

First we need to assume some boundedness. Henceforth, let us assume throughout that $\abs{Y}\leq 1$ and that $\| h \|_\infty \leq 1$ for every $h\in\Hcal$.

Second we need to characterize the complexity of our function class $\Hcal$. Let $\mathcal{D}_v$ be the set of observations belonging to fold $v$, that is $\mathcal{D}_v = \{ O_i : i \in [N],~ V_i  = v \}$. Let $\epsilon_1,\ldots, \epsilon_N$ be i.i.d. Rademacher random variables independent from $O_1,\ldots, O_N$. For any function class $\mathcal{F}$, we introduce the following Rademacher complexities:
\begin{align}
\mathcal{R}_{N}(\mathcal{F}, \delta) =& \EE \left[ \sup_{f \in \mathcal{F}: \| f \| \leq \delta} \frac{2}{N} \sum_{i : V_i = 0 } \epsilon_i f(S_i) \right], \\
\mathcal{R}_{n,a}(\mathcal{F}, \delta) =& \EE \left[ \sup_{f \in \mathcal{F}: \| f \|_{(a)} \leq \delta} \frac{2}{n} \sum_{i:A_i= a, V_i = 0} \epsilon_i f(S_i) \right],\\
\text{and} \qquad \mathcal{R}_{N}^{|0}(\mathcal{F}, \delta) =& \EE \left[ \sup_{f \in \mathcal{F}: \| f \| \leq \delta} \frac{2}{N} \sum_{i : V_i = 1} \epsilon_i f(S_i) \mid \mathcal{D}_0 \right].
\end{align}

The relevant function classes for analyzing npJIVE are:
$\Hcal^0=\Hcal-h_0$, $\Gcal^0=\{g(S)-(\T_K g)(A):g\in\Hcal^0\}$, $\Gcal^1=\{g(S)(\T_K g)(A):g\in\Hcal^0\}$, and $\widehat\Gcal=\{g(S)(\widehat\T_{K,0} g)(A):g\in\Gcal^0\}$.

Let $\mathrm{star}(\mathcal{F})=\{\gamma g:\gamma\in[0,1],g \in \mathcal{F}\}$. Let $\widetilde \delta_N$, $\bar \delta_N$ and $\breve \delta_N$ be solutions to $\mathcal{R}_N(\mathrm{star}(\mathcal{F}), \delta) \leq \delta^2$ for $\mathcal{F} = \Gcal^0, \Gcal^1, \T_K \Gcal$, respectively, let $\widetilde \delta_n$ be a solution to $\mathcal{R}_{n,a}(\mathrm{star}(\Gcal^0), \delta) \leq \delta^2$ for every $a$, and let $\widehat \delta_N$ be a solution to $\mathcal{R}_{N}^{|0}(\mathrm{star}(\widehat \Gcal), \delta) \leq \delta^2$. Let $\widehat r_N^2$ and $\bar r_N^2$ be bounds on $\mathcal{R}^{|0}(\widehat \Gcal, \infty)$ and $\mathcal{R}(\Gcal^1, \infty)$.

Here, $\widehat r_N^2,\bar r_N^2$ are (unlocalized) Rademacher complexity measures and $\widetilde \delta_{n},  \widetilde \delta_N, \widehat \delta_N,\bar \delta_N, \breve \delta_N$ are critical radii, both of which are standard measures of generic functional complexity (see \citep{wainwright2019high} which also gives bounds on these quantities for a variety of classes including linear, reproducing kernel Hilbert spaces, H\"older or Sobolev spaces, neural nets, \etc).
Finally, let $\delta_N = \max(\widetilde \delta_N, \widehat \delta_N,\bar \delta_N, \breve \delta_N)$ and $r_N^2 = \max( \widehat r_N^2,\bar r_N^2)$.

\begin{theorem}
    Let $\zeta\in(0,0.5)$ be given.
    Suppose that $\delta_N^2 = o(\lambda)$, $N^{-1/2} \sqrt{\log \log \widehat \delta_N} = o(\widehat \delta_N)$, $N^{-1/2} \sqrt{\log \log \bar \delta_N} = o(\bar \delta_N)$, $n^{-1/2} \sqrt{\log \log \widetilde \delta_n} = o(\widetilde \delta_n)$ and that $\sqrt{\log (K / \zeta)} = o(\widetilde \delta_n)$. Then,
    \begin{align*}
    &\| \T_K (\widehat h - h_0) \|^2 = O(\lambda + \delta_N \delta_n)~~\text{with probability $1-O(\zeta)$, and}\\
    &\| \widehat h - h_0 \| = o_P(1)
    \end{align*}
\end{theorem}

\section{An application of npJIVE: Confounding Robust Surrogate Indices from Many Weak Experiments}\label{sec: surrogates}

We now turn to the problem of using many historical experiments to construct surrogate indices for long-term causal inference in a novel experiment.
The long-term effect of interventions is often of primary concern in causal inference. Examples include the effect of early-childhood education on lifetime earnings \citep{chetty2011does}, of promotions on long-term value \citep{yang2020targeting}, and of digital platform design on long-term user retention \citep{hohnhold2015focusing}. While the gold standard for causal inference is experimentation, the significant delay of long-term observations after assignment to treatment means that, even when we can randomize the intervention, we may not be able to measure the outcome of interest.

This presents significant challenges and may even alter incentives such that scientists prioritize interventions that can be evaluated with short-run outcomes, as documented by \citet{budish2015firms} in the case of cancer drugs.
Nevertheless, other relevant post-treatment outcomes are often available in the short-term. For example, in AIDS treatments we observe short-term viral loads or CD4 counts well before mortality outcomes \citep{fleming1994surrogate}. Similarly, in digital experimentation, we observe short-term signals on user engagement well before retention or revenue shifts.

An appealing approach to leverage these short-term observations is in the construction of a surrogate index, whereby we impute unobserved long-term outcome using their prediction from multiple short-term surrogate observations.

\subsection{Background: Statistical-Surrogate Indices}

Possibly the most common approach assumes that short-term observations form a \emph{statistical surrogate} \citep{prentice1989surrogate}. This assumes that the long-term outcome is conditionally independent of treatment given the short-term outcomes. This involves two key restrictions: that there are no unobserved confounding between short- and long-term outcomes (no $U_1$ in \cref{fig:simple}) and that all of the treatment's effect on long-term outcomes is mediated by the short term (no $U_2$ in \cref{fig:simple}).
As the latter restriction becomes more defensible as we include more short-term outcomes so as to mediate more of the treatment's long-term effect, \citet{athey2019surrogate} combine many short-term observations into a surrogate index, assuming they form a statistical surrogate and using historical data to regress long-term on short-term (or other ways of adjustment such as weighting). The proposal is simple, effective, and, as such, widely adopted.

However, even if short-term outcomes fully mediate the long-term effect, they may fail to satisfy statistical surrogacy. Consider the causal diagram in \cref{fig:simple}, where $S$ perfectly mediates $A$'s effect on $Y$ (\ie, exclusion restriction), but $S$ and $Y$ share an unobserved confounder $U$, while treatment $A$ is fully unconfounded. In this case, $S$ is a collider so that conditioning on it induces a path from $A$ to $Y$ via $U$, violating surrogacy and imperiling analysis using such methods as \citet{athey2019surrogate}. 

This scenario is actually exceedingly commonplace: for example, in streaming platforms with a subscription model, a user's amount of free time both impacts their short-term engagement and their probability of subscription retention in the same direction. Failing to address this confounding leads to surrogate-based estimates that overestimate true effects on long-term outcomes. We might even have more extreme situations in which an intervention strongly increases short-term engagement of a subpopulation of users who are unlikely to unsubscribe while slightly decreasing engagement of a subpopulation of users very likely to unsubscribe. This might result in both an overall increase in short-term engagement and an overall decrease in long-term retention, a situation known as the surrogate paradox \citep{elliott2015surrogacy}.

\subsection{Experiments as Instruments and Surrogates as Proxies}

\citet{athey2019surrogate} consider a setting where historical data prior to the present experiment contains only $S$ and $Y$ and some baseline covariates. In this context, all we can do is either worry about potential unobserved confounders, or hope the included covariates satisfy ignorability between $S$ and $Y$. However, at organizations that routinely run many digital experiments, we can take historical data from past experiments where we also observe the randomized treatments $A$. In the setting of \cref{fig:simple}, these treatments constitute an IV, which can help us identify the causal effect of $S$ on $Y$ and therefore infer the effect of a novel treatment on $Y$ by considering only its effect on $S$. 

Going beyond this, we can even account for some exclusion violations (a long-term effect unmediated by short-term observations) under additional structure. In the following we show that some parts of $S$ can serve as effect mediators, while other parts of $S$ can serve as negative controls (also known as proxies) that can adjust for effects unmediated by the first part, so long as we get a rich enough view onto them.

In the following, we first present the corresponding causal model (summarized in \cref{fig:complex}). Then, we present an identification result showing how the long-term outcome of a novel treatment can be identified without observations on long-term outcomes under the novel treatment by leveraging historical experiments where we observe both short- and long-term outcomes. This identification result is written as a function of NPIV solutions, which may be estimated slowly in this nonparametric setting, imperiling rates on causal effect estimation if we were to simply plug-in an estimated solution. Therefore, we finally present a debiased formulation of the identification result that, under a strong identification condition, enjoys a mixed-bias property, which means we can multiply estimation rates and make them insignificant to first order. Together with estimation results for a nuisance in this new debiased formulation, we obtain asymptotic normal estimates for long-term causal effects of the novel treatment, enabling inference.

\begin{figure}[b]
\centering
\caption{Causal diagrams for surrogate settings with unobserved confounders. Dashed circles ($U$) indicate unobserved variables. Dotted circles ($Y$) indicate variables observed historically, but unobserved for novel treatments.}
\centering
\begin{tikzpicture}
\node[draw, circle, text centered, minimum size=0.75cm, line width= 1] (a) {$A$};
\node[draw, circle, right=1 of a, text centered, minimum size=0.75cm, line width= 1] (s) {$S$};
\node[draw, circle, above right=0.5 and 0.325 of s,text centered, minimum size=0.75cm, dashed,line width= 1] (u) {$U$};
\node[draw, circle, right=1 of s, text centered, minimum size=0.75cm, dotted, line width= 1] (y) {$Y$};
\draw[-latex, line width= 1] (a) -- (s);
\draw[-latex, line width= 1] (s) -- (y);
\draw[-latex, line width= 1] (u) -- (s);
\draw[-latex, line width= 1] (u) -- (y);
\end{tikzpicture}\vspace{1em}
\caption{A setting with unconfounded treatment ($A$) but confounded surrogate ($S$) and outcome ($Y$).}\label{fig:simple}

\end{figure}

\subsection{Causal Model}

Up until now we worked only with an observed-data distribution and posed \cref{eq: npiv} as a statistical problem. We now introduce a causal model for the surrogate setting we study, for which we will show how the NPIV problem \cref{eq: npiv} helps identify a causal effect.

\paragraph{Potential outcomes.}
We state our causal model in general terms using potential outcomes. The model is summarized in \cref{fig:simple}, which represents a causal diagram that is consistent with these assumptions. 
We assume the existence of random variables $\widetilde A, U, S(a), Y(s)$ and we refer to the variables $S(a)$ and $Y(s)$ as potential outcomes. We define $Y(a) = Y(S(a))$.

Here we let $ \widetilde A\in\{1,\dots,K\} \cup \{ \new \}$, where $\widetilde A \in \{1, \ldots, K\}$ denotes the historical experimental data and $\widetilde A=\new$ denotes the novel treatment (we will return to data generation after the causal model).
We are interested in inference on the average outcome under the novel treatment:
$$
\theta_0=\EE[Y(\new)].
$$

Our key assumptions characterizing the causal model are:
\begin{assumption}[Structural relationship]\label{asm:structural_relationship}
    There exists a function $h^\star$ such that $\forall a \in \{0,\ldots,K\}$ $E[h^\star(S(a))] = E[Y(a)]$.
\end{assumption}

\begin{assumption}[Unmeasured confounder and potential outcomes independence]\label{asm:indep_pot_outcomes} For every $s$ and $a$,
$Y(s) \indep S(a) \mid U$.    
\end{assumption}

\begin{assumption}[Randomization]\label{asm:randomization}
    $\widetilde A$ is independent of all potential outcomes and of $U$.
\end{assumption}

\Cref{asm:indep_pot_outcomes} captures the presence of confounding between short- and long-term outcomes. \Cref{asm:randomization} captures that experiments are randomized so treatment assignment can be used as an IV. 

\paragraph{Example.} It is straightforward to check that our causal assumptions hold if the potential outcomes are generated by the nonparametric structural causal model (SCM)
\begin{align*}
    Y =& f_Y(S, \epsilon_Y) + g_Y(U, \epsilon_Y) \\
    S =& f_{S}(\widetilde A, U,\epsilon_{S}) \\
    U=& f_{U}(\epsilon_{U}) \\
    \widetilde A =& f_A (\epsilon_A)
\end{align*}
where $\epsilon_Y$, $\epsilon_S$, $\epsilon_U$, and $\epsilon_A$ are independent. The key for assumption \ref{asm:structural_relationship} to hold is that $U$ and $S$ enter the outcome equation in an additively separable manner.

\subsection{Data}

Let $\widetilde S=S(\widetilde A)$ and $\widetilde Y=Y(\widetilde A)$ be the factually observed outcomes

\paragraph{Historical data set.} Let $A$ be uniformly distributed over $[K]$, let $Y, S$ be such that $(Y,S) \mid A = a \sim (\widetilde Y, \widetilde S) \mid \widetilde A = a$. Let $O = (A,S,Y)$.
We have $N$ historical units for which we have observations $O_i = (A_i, S_i, Y_i), i=1,\ldots,N$ such that $O_i \sim O$.
The treatment assignments $A_1\ldots, A_N$ of the historical units $i=1,\ldots,N$ is drawn following a completely randomized design such that each cell has exactly $n$ units. Any two $i \neq j$, $(S_i,Y_i)$, $(S_j,Y_j)$ are independent conditional on $A_i, A_j$. 

\paragraph{Novel data set.} Let $S^{\new} \sim \widetilde S \mid \widetilde A = 0$. We observe $n'$ independent copies $S^{\new}_1,\ldots, S^{\new}_{n'}$ of $S^{\new}$.

We now turn to the question of identification.
Note that had we observed $Y^{\new} \sim \widetilde Y \mid \widetilde A = 0$ from the novel treatment, we would have trivially identified $\theta_0$ as $\EE[Y^{\new}]$ under \cref{asm:randomization}. So, in the absence of outcome observations for the novel treatment, our goal is to combine the historical and novel datasets to identify $\theta_0$.

\subsection{Surrogacy and set identification}

It is immediate under the structural relationship (assumption \ref{asm:structural_relationship}) that $h^\star(S)$ is a valid surrogate index.

That is, $h^\star$ connects causal effects on $Y$ to causal effects on $S$: if we could access $h^\star$, we would be able to know the average potential outcome of $Y$ under any treatment from the distribution of the average potential outcome of $S$ under that same intervention. We thus define the causal parameter of interest as $\theta^\star =  \theta(h^\star)$ where $\theta(h) = \EE[h(S(\new))]$ for the novel intervention $0$.

We consider a nuisance space $\mathcal{H} \subseteq L_2(\mathcal{S})$. Let $\mathcal{T}_K: \mathcal{H} \to ([K] \to \mathbb{R})$ be the operator defined for every $h \in \mathcal{H}$ and every $a \in [K]$ by $(\mathcal{T}_K h)(a) = \EE[ h(S) \mid  A = a]$.

\paragraph{Set identification of $h^\star$ and point identification of $\theta_0$.} 

The following lemma characterizes $h^\star$ in terms of the data distribution in the historical cells.

\begin{lemma}[Set identification of $h^\star$]\label{lemma:set_id_h_star}
    Under the causal assumptions \ref{asm:structural_relationship}-\ref{asm:randomization}, it holds that $h^\star \in \mathcal{H}_0$ where $\mathcal{H}_0$ is the set of functions $h: \mathcal{S} \to \mathbb{R}$ that satisfy the conditional moment restriction (CMR) $\EE[h^\star(S) \mid A =a] = \EE[Y \mid A=a]$ for every $a \in [K]$.
\end{lemma}

\subsection{Identification and strong identification}

While $h^\star$ is not point identified a priori, it is not our parameter of interest: $\theta(h^\star)$ is. We provide below a formal definition of what it means for $\theta(h^\star)$ to be identified.

\begin{definition}[Identification]
 Let $r_{0,K}$ be the $[K] \to \mathbb{R}$ function defined for every $a \in [K]$ by $r_{0,K}(a) = E[Y \mid A = a]$.
    We say that  $\theta(h^\star)$ is identified if for any $h \in \mathcal{H}$ such that $\mathcal{T}_K h = r_{0,K}$, that is for any $h \in \mathcal{H}_{0,K} \cap \mathcal{H}$, $\theta(h) = \theta(h^\star)$.
\end{definition}

Let $\rho$ be the importance sampling ratio between the distribution of $S$ in the target treatment arm ($A=0$) and in the historical experiments ($A \in [K]$), that is, for every $s \in \mathcal{S}$,
\begin{align}
    \rho_K(s) = \frac{p_{S(\new)}(s)}{\frac{1}{K} \sum_{a=1}^K p_{S(a)}(s)} = \frac{p_{S^{\new}}(s)}{p_S(s)}.
\end{align}
For any $h \in \mathcal{H}$ we can rewrite $ \theta(h)$ as an expectation under the historical experiments data in terms of the importance sampling ratio: 
\begin{align}
    \theta(h) = \EE[\rho_K(S) h(S)] = \EE[ \alpha_K(S) h(S) ] \qquad \text{where} \qquad \alpha_K = \Pi ( \rho_K \mid \mathcal{H}),
\end{align}
where $\Pi( \cdot \mid \mathcal{H})$ is the $L_2$ projection on $\mathcal{H}$.
The function $\alpha_K$ is then the Riesz representer of the functional $\theta$. It is then immediate that $\theta(h^\star)$ is point identified if, and only if, $\alpha_K \in \mathcal{N}(\mathcal{T}_K)^\perp = \overline{\mathcal{R}(\mathcal{T}_K^*)}$. Note that identification in our setting is a strong condition: indeed, $\mathcal{N}(\mathcal{T}_K)^\perp$ is a $K$-dimensional subspace of $\mathcal{H}$, which places considerable restrictions on the allowable complexity of $\alpha_K$. In the asymptotic regime we consider, $K \to \infty$, which mitigates asymptotically this stringent condition. In the next subsection, we consider an approximate notion of identification which might be more realistic in our setting. Before presenting this notion, we nevertheless discuss how identification relates to another, stronger notion of identification discussed in various recent articles [ref needed].

\begin{definition}[Strong identification]
    We say that $\theta(h^\star)$ is strongly identified if $\alpha_K \in \mathcal{R}(\T_K^* \T_K)$, that is if $\Xi_K := \{ \xi \in \mathcal{H} : \T_K^* \T_K \xi = \alpha_K \} \neq \emptyset$.
\end{definition}
Note that strong identification trivially implies identification. Under strong identification, $\theta(h^\star)$ admits an additional representation, as we make explicit in the next lemma.

\begin{lemma}[Double robustness, mixed bias]
    Let $S'$ be such that $S' \indep (S, Y) \mid A$ and $S' \mid A \sim S \mid A$. Suppose that $\theta(h^\star)$ is strongly identified. Let $h, \xi \in \mathcal{H}$ and let $h_{K} \in \mathcal{H}_{0,K} \cap \mathcal{H}$ and  $\xi_K \in \Xi_K$. Let   
    \begin{align}
        \psi(h, \xi) = h(S^\new) + \xi(S')(Y - h(S)).
    \end{align}
    Then it holds that 
    \begin{align}
        \left\lvert \EE [ \psi(h, \xi) ] - \theta^\star \right\rvert \leq \left\lVert \T_K(\xi - \xi_K) \right\rVert \times \left\lVert \T_K (h - h_K) \right\rVert,
    \end{align}
    where for any $f: [K]\times  \mathcal{S} \times \mathcal{Y} \to \mathbb{R}$, $\left\lVert f \right\rVert^2 = \EE [ f(A,S,Y)^2 ]$, that is $\| \cdot \|$ is $L_2$ norm induced by the historical data distribution.
\end{lemma}

We refer to $\xi_K$ above as the \textit{debiasing} nuisance and to $h_K$ as the primary nuisance. As we will see in section \ref{sec: one-step}, the mixed bias property implies that we can get parametric-rate inference for $\theta^\star$ from nonparametrically estimated nuisances as long as the product of their weak norms converges to zero at a parameteric rate.

\begin{remark}
    The requirement that $S'$ is an independent copy of $S$ will be easily satisfied at estimation time by splitting treatment arms in non-overlapping folds and drawing $S'$ for the same cell as $(S,Y)$ but from a different fold.
\end{remark}

\begin{remark}
    Unlike previous works, we get a mixed bias property in terms of weak norms of both $h$ and $\xi$. This is very specific to our setting though, as the key to this property is that we can easily draw an independent copy $S'$ of $S$, which we can do here due to (i) the fact that we can group observations by value of $A$ (that is in treatment arms) and split these groups in folds, and (ii) the non-overlapping nature of the variables $A$, $S$, $Y$.
\end{remark}

We now turn to the key result of this section.

\begin{theorem}[Identification implies strong identification]\label{thm:id_implies_strong_id}
    In the setting we consider in this section, identification and strong identification are equivalent, that is 
    \begin{align}
        \overline{\mathcal{R}(\T_K^*)} = \mathcal{R}(\T_K^*) = \mathcal{R}(\T_K^* \T_K).
    \end{align}
\end{theorem}

Thefore, if $\theta^*$ is identified, the existence of a debiasing nuisance and the double robust representation is automatically guaranteed.

\subsection{Approximate strong identification}

In this subsection, we provide a notion of approximate identification which accounts for violations of the stringent identification condition of the previous section. 

\begin{definition}
    We say that $\theta(h^\star)$ is $(K, \delta)$-approximately strongly identified if 
    $\min_{q \in \mathcal{Q}_K} \| \alpha_K - \T_K^* q \| \leq \delta$.
\end{definition}

\begin{remark}
    Using the same arguments as in the proof of theorem 4, we can show that $(K,\delta_K)$-strong identification is equivalent to $\inf_{\xi \in \mathcal{H}} \| \alpha_K - \T_K^* \T_K \xi \| \leq \delta_K$, which makes more explicit the connection with strong identification.
\end{remark}

\begin{remark}
    As $\mathcal{R}(\T^*_K)$ grows with $K$, the sequence of $(\delta_K)$ of smallest values of $\delta$ for which approximate strong identification holds is non-increasing. Intuitively, it should decrease faster if the historical experiments impose more ``complementary'' moment restrictions. 
\end{remark}

The following theorem shows that under approximate strong identification, $\theta(h^\star)$ is identified up to an error that decreases with $K$.

\begin{theorem}[Approximate target identification]
    If $\| \Pi( h^\star \mid \mathcal{N}(\T_K)) \| \leq \epsilon_K$, $(K, \delta_K)$-approximate strong identification holds and $q_K \in \argmin \| \alpha_K - \T_K^* q \|$ and $h_K^\dagger$ is the minimum norm element of $\mathcal{H}_{0,K} \cap \mathcal{H}$, then, for 
        $\psi(h, q) = h(S^\new) + q(A) (Y - h(S))$,
        we have
        \begin{align}
            | \EE [ \psi(h_K^\dagger, q_K)] - \theta(h^\star) | \leq \epsilon_K \delta_K.
        \end{align}
\end{theorem}

We obtain the following appproximate version of the mixed-bias result as a corollary. 

\begin{corollary}
    Let $h \in \mathcal{H}$, $q \in \mathcal{Q}$. Let $q_K^\dagger$ be the minimum norm element of $\argmin_{q \in \mathcal{Q}} \| \alpha_K - \T_K^* q\|$. Under the assumptions and notation of the previous theorem,
    \begin{align}
        | \EE \psi(h, q) - \theta(h^\star) | \leq \epsilon_K \delta_K  + \delta_K \| h - h_K^\dagger \| + \| q - q_K^\dagger \| \| \T_K (h - h^\star) \|.
    \end{align}
\end{corollary}

\paragraph{Representation of $q_K^\dagger$.} While the above definition of $q_K^\dagger$ shows how to identify it from the historical data distribution and the novel data distribution, doing so requires computing several density ratios: $\alpha_K$ is the projection of an importance sampling ratio, and the adjoint operator $\T_K^*$ can be shown to involve the conditional densities of $S$ given values of $A$ in the historical data. As density ratios are typically hard to estimate this doesn't suggest a straightforward estimation strategy. We now show that we can obtain $q_K^\dagger$ by solving two consecutive quadratic optimization problems that have easy-to-implement sample-level counterparts. We have that $\mathcal{Q}$ is the linear span of the functions $q_a^* := \T_K^* \bm{1}\{\cdot = a\}$ for $a = 1,\ldots, K$. Using the adjoint property, we have that
\begin{align}
    q_a^* = \argmin_{h \in \mathcal{H}} \EE[h^2(S)] - 2 E[h(S) \mid A = a]. \label{eq:approx_id_basis}
\end{align}
Finding $q_K^\dagger$ now reduces to finding 
\begin{align}
    \bm{\gamma} = \argmin_{\bm{\gamma} \in \mathbb{R}^K} E[(\sum_{a=1}^K \gamma_a q_a^*(S))^2] - 2 E[\alpha_K(S) \sum_{a=1}^K \gamma_q q_a^*(S) ],
\end{align}
that is, using the representer property of $\alpha_K$, as in automatic debiased machine learning \citep{chernozhukov2022automatic},
\begin{align}
     \bm{\gamma} =\argmin_{\bm{\gamma} \in \mathbb{R}^K} \sum_{1 \leq a,a' \leq K} \gamma_a \gamma_{a'} \EE [q_a^*(S) q_{a'}^*(S)] - 2 \sum_{a=1}^K \gamma_a E[q_a^*(S^\new)]. \label{eq:approx_id_OLS}
\end{align}
We discuss how to translate the above at the sample level for estimation in subsection \ref{sec: approx_id_nuisance_estimation} of the next section. 

\section{Estimation of debiasing nuisances}\label{sec: estimation_debiasing_nuisance}

Estimation of the primary nuisance $h$ is a direct application of the npJIVE results of section \ref{sec: npjive}. In this section we show how to estimate the debiasing nuisances under both identification and approximate identification. Under identification, the strategy and the analysis closely follows that of the primary nuisance. 
In both cases, we provide cross-folded estimator and/or Jackknifed estimators.

\subsection{Estimation of the debiasing nuisance function under (strong) identification}\label{sec: semiparam}

In this section we suppose that $\theta^\star$ is identified, as defined in section \ref{sec: surrogates}. From theorem \ref{thm:id_implies_strong_id}, there exists a debiasing nuisance $\xi_K$ that is $\T_K^* \T_K \xi_K = \alpha_K$. Observe that $\T_K^* \T_K \xi_K = \alpha_K$ 
is the first order condition of the optimization problem $\min_{\xi \in \mathcal{H}} R_2(\xi)$ where
\begin{align}
    R_2(\xi) = \EE\left[ [\T_K \xi](A)^2 \right] - 2 \EE \left[ \alpha_K(S) \xi(S) \right].
\end{align}
From the definition of $\alpha_K$, for any $\xi \in \mathcal{H}$, we can rewrite $R_2(\xi)$ as
\begin{align}
    R_2(\xi) = \EE\left[[\T_K \xi](A)^2  \right] - 2 \EE\left[ \xi(S^\new) \right],
\end{align}
which avoids reference to typically hard to estimate density ratios, as in automatic debiased maching learning \citep{chernozhukov2022automatic}. 

As in the case of the primary nuisance, the risk for the debiasing nuisance involves the quadratic projected term $\EE\left[ [\T_K \xi](A)^2 \right]$, which, as we argued in section \ref{sec: npjive} cannot be estimated consistently by $K^{-1} \sum_{a=1}^K  [\widehat\T_K \xi](a)$ if $n$ stays bounded. To overcome this issue, as in section \ref{sec: npjive}, we introduce a Tikhonov regularized cross-fold empirical analog of the population debiasing risk. Specifically, for any regularization level $\mu > 0$, we let 
\begin{align}
    \widehat R_{2, \mu}(\xi) =  \frac{1}{2K} \sum_{a=1}^K [\widehat \T_{K,0} \xi](a) [\widehat \T_{K,1} \xi](a) - \frac{1}{n_\new} \sum_{A_i = a'} \xi(S^\new_i) + \mu \| \xi \|_{2,N}^2.
\end{align}
Let $\widehat \xi$ be a minimizer over $\mathcal{H}$ of the regularized empirical risk $\widehat R_{2, \mu}(\xi)$, and let $\xi_K^\dagger$ be the minimum norm element of the set $\Xi_K$ of debiasing nuisances. We present next convergence rate guarantees for $\widehat \xi$. We start with a slow rate result. Before stating it, we need to introduce an additional measure of complexity: let 
\begin{align}
    \mathcal{R}^\new(\mathcal{H}, \delta) = \EE \left[\sup_{\xi \in \mathcal{H} : \|\xi\|_\new \leq \delta} \frac{1}{n_\new} \sum_{i=1}^{n_\new} \epsilon_i \xi(S_i^\new) \right]
\end{align} be the localized Rademacher complexity of $\mathcal{H}$ on the new dataset ($A = \new$). Let $(r^\new_{n_\new})^2$ be an upper bound on  $ \mathcal{R}^\new(\mathcal{H}, 1)$.
Without any further assumptions, we obtain the following slow rate result.

\begin{theorem}[Slow rate for $\widehat \xi$]\label{asm:slow_rate}
Let $\zeta > 0$.  Suppose that $\sqrt{\log (K / \zeta)} = o(\widetilde \delta_n)$, and $\sqrt{\log \log \delta_N / N} = o(\delta_N)$.
    Then it holds with probability $1 - O(\zeta)$ that
    \begin{align}
        \| \T_K (\xi - \xi_0) \| = \widetilde O(\delta_N^{1/2} + r^\new_{n_\new} + \sqrt{\mu}).
    \end{align}
\end{theorem}
We show next that, under an additional assumption, we can obtain a fast rate in that replaces the unlocalized Rademacher complexity $r^\new_{n_\new}$ with a critical radius of $\mathcal{H}$ over the new data set. The additional assumption we require is a so-called source condition on the Riesz representer $\alpha_K$, initially introduced in \cite{bennett2023source}. We state it below.
\begin{definition}[Source condition]\label{def:source}
    We say that the source condition with parameter $\beta$ holds for $\alpha_K$ if 
    $\alpha_K = (\T_K^* \T_K)^\beta w_K$ for some $w_K \in \mathcal{H}$.
\end{definition}
For $\beta = 1$, the source condition reduces to the strong identification condition we introduced in section $\ref{sec: surrogates}$.
As we will make explicit in the next theorem, the source condition buys us additional convergence speed if $\beta > 1$. Let $\delta^\new_{n_\new}$ be a critical radius of $\mathcal{H}$ over the new data set defined as a  solution to $\mathcal{R}^\new(\mathcal{H}, \delta) \leq \delta^2$. 

\begin{theorem}[Fast rate]\label{asm:fast_rate}
    Let $\zeta > 0$.  Suppose that $\sqrt{\log (K / \zeta)} = o(\widetilde \delta_n)$, and $\sqrt{\log \log \delta_N / N} = o(\delta_N)$. Further suppose that 
    the source condition holds with parameter $\beta \in (1, 2]$ and that $w_K$ in the source condition is uniformly bounded in $K$, that is $\sup_K \| w_K \| < \infty$. Set $\mu = (\delta^\new_{n_\new})^{\frac{2}{2 + \beta}}$.
    Then it holds with probability $1 - O(\zeta)$ that
    \begin{align}
        \| \T_K (\widehat \xi - \xi_0) \| = \widetilde O((\delta^\new_{n_\new})^{-\frac{1 + \beta}{2 + \beta}}).
    \end{align}
    and 
    \begin{align}
        \| \widehat \xi - \xi_K \| = \widetilde O((\delta^\new_{n_\new})^{-\frac{\beta}{2 + \beta}}).
    \end{align}
\end{theorem}
Note that the strong norm convergence rate above is in terms of the norm induced by the historical data distribution.

\subsection{Estimation of the debiasing nuisance function under approximate identification}\label{sec: approx_id_nuisance_estimation}

In this section we propose an estimator of the debiasing nuisance $q_K^\dagger$ under approximate identification. While we leave guarantees to follow-up work, we heuristically justify its construction in this subsection and empirically validate its performance in section \ref{sec: num_exps}. 

Based on the representation of $q_K^\dagger$ as the solution to two consecutive optimization problems \eqref{eq:approx_id_basis} and \eqref{eq:approx_id_OLS}, a natural estimation strategy is to define, for every $a \in [K]$,
\begin{align}
    \widehat{q}_a^* = \argmin_{h \in \mathcal{H}} \frac{1}{N} \sum_{i=1}^N h(S_i)^2 - 2 \frac{1}{n} \sum_{i : A_i = a} h(S_i), \label{eq:approx_id_emp_basis}
\end{align}
and let $\widehat{q} = \sum_{a=1}^K \widehat \gamma_a \bm{1}(\cdot = a)$ where
\begin{align}
    \widehat{\bm{\gamma}} = \argmin_{\bm{\gamma} \in \mathbb{R}^K} \sum_{1 \leq a, a' \leq K} \gamma_a \gamma_{a'} \frac{1}{N} \sum_{i=1}^N \widehat{q}_a^*(S_i) \widehat{q}_{a'}^*(S_i) - 2 \sum_{a=1}^K \gamma_a \frac{1}{n_\new} \sum_{i=1}^{n_\new}  \widehat{q}_a^*(S_i^\new). \label{eq:approx_id_emp_OLS}
\end{align}
For the same reason we already mentioned earlier, the above risk isn't consistent under weak-instrument asymptotics: since $\widehat{q}_a^*$ optimizes a risk with an empirical mean over $n$ terms, $\| \widehat{q}_a^* - q_a^* \|^2$ doesn't shrink as $n$ stays fixed and appears in the expectation on the empirical risk in \eqref{eq:approx_id_emp_OLS}. 

We therefore propose a Jackknifing device to eliminate squares of errors that do not converge to zero as $n$ stays fixed. For every $a \in [K]$ and $i$ in historical treatment arm $a$ (that is such that $A_i = a$), we define 
\begin{align}
    \widehat{q}_{a,-i}^* = \argmin_{h \in \mathcal{H}} \frac{1}{N} \sum_{j=1}^N h(S_j) - \frac{2}{n-1} \sum_{j: \substack{A_j = a \\ j \neq i} } h(S_j).
\end{align}
Now observe that for any $a$, $\EE [q_a^*(S)^2] = \E[[\T_K q_a^*](A) q_a(A)] = K^{-1} \EE[ q_a^*(S) \mid A=a]$. This suggests estimating the expectations of squares in \eqref{eq:approx_id_OLS} with $K^{-1} n^{-1} \sum_{i: A_i = a} \widehat{q}_{a,-i}^*(S_i)$. Concerning the cross terms $\EE [q_a^*(S) q_{a'}^*(S)]$ for $a \neq a'$, the treatment-arm-level means that lead to $\widehat{q}_a$ and $\widehat q_{a'}^*(S)$ are over different observations, so these terms shouldn't cause issues. This discussion suggests defining
$\widehat{q} = \sum_{a=1}^K \widehat \gamma_a \bm{1} \{  \cdot = a\}$ where 
\begin{align}
    \bm{\gamma} = \sum_{1 \leq a, a' \leq K} \gamma_a \gamma_{a'} C_{a,a'} - 2 \sum_{a=1}^K \gamma_a \frac{1}{n_\new} \sum_{i=1}^{n_\new}  \widehat{q}_a^*(S_i^\new),
\end{align}
with \begin{align}
    C_{a,a'} = \begin{cases}
        \frac{1}{K n} \sum_{i : A_i = a} \widehat{q}_{a,-i}^*(S_i) & \text{ if } a = a', \\
        \frac{1}{N} \sum_{i=1}^N \widehat{q}_a^*(S_i)  \widehat{q}_{a'}^*(S_i)  & \text{ if } a \neq a'.
    \end{cases}
\end{align}

\section{Semiparametric cross-fold estimator}\label{sec: one-step}

Divide the historical data set in 4 folds $v=0,1,2,3$. Estimate $\widehat h$ on the first two folds and $\widehat \xi$ on the first two folds and the new treatment arm. For every index $i$ in fold $2$, let $j(i)$ be a corresponding observation index in fold 3, such that $i \mapsto j(i)$ is a one-to-one mapping, and let $(S_i', Y_i') = (S_{j(i)}, Y_{j(i)})$\footnote{this requires that the two fold have the same number of observations}.
\begin{align}
    \widehat \theta = \frac{1}{n'} \sum_{i=1}^{n'} \widehat h(S_i^{\new}) + \frac{4}{N} \sum_{i} \xi(S_i) (Y_i' - \widehat h(S_i')).
\end{align}

\begin{theorem}
    Suppose that $\| \T_K ( \widehat h - h_0 ) \|  \| \T_K ( \widehat \xi - \xi_0 ) \|= o_P( (n' \wedge N)^{-1/2})$ and that $\| \widehat h - h_0 \| = o_P(1)$ and $\| \widehat \xi - \xi_0 \| = o_P(1)$ .
    Then 
    \begin{align}
        \sqrt{n' \wedge (N/4)} (\widehat \theta - \theta(h_0)) \xrightarrow{d} \Ncal(0,\sigma),\quad
        \sigma=\begin{cases}
            \sigma_1^2 & \text{if } n' / N \to 0, \\
            \sigma_1^2 + c \sigma_2^2 & \text{if } 4n' / N \to c \in (0, \infty), \\
            \sigma_2^2 & \text{if } n' / N \to \infty,
        \end{cases}
    \end{align}
    where $\sigma_1^2 = \mathrm{Var}(h_0(S) \mid A=a')$ and $\sigma_2^2 = \mathrm{Var}(\xi_0(S) \mid A ) \times \mathrm{Var}(Y - h_0(S) ) $.
\end{theorem}

\section{Numerical Experiments}\label{sec: num_exps}

In this section, we consider the confounded surrogate setting of section \ref{sec: surrogates}. We demonstrate inconsistency of minimum distance estimators and consistency of npJIVE under many-weak-IVs asymptotics, and and debiasing via the one-step estimator.

We consider a one dimensional vector of short-term variables $S$. We let the structural parameter $h^\star$ be the function $h^\star : s \mapsto s + \sin(s) + \bm{1}\{ s > 0.25 \}$. For any given $K$, we sample the first-stage effects $\Gamma_1$,\ldots,$\Gamma_K$ on $S$ as i.i.d. draws from a normal $\mathcal{N}(0, \sigma_\Gamma^2)$. For $i = 1,\ldots,N$, we let $A_i = \left\lfloor  i / n \right\rfloor$. We draw $S_i$ and $Y_i$ according to the following set of structural equations:
\begin{align}
    U_i \sim& \mathcal{N}(0, \sigma_U^2), \\
    S_i \sim& \mathrm{Uniform}(\Gamma_{A_i} - 0.5, \Gamma_{A_i} + 0.5) + U_i, \\
    Y_i =& h(S_i) - U_i.
\end{align}
The random variable $U_i$ is an unmeasured confounder. The draws $U_1,\ldots,U_N$ are independent. 
We generate independent draws $S_1^\new, \ldots, S^\new_{n^\new}$ according to 
\begin{align}
    U_i^\new \sim& \mathcal{N}(0, \sigma_U^2),\\
    S_i^\new \sim& \mathrm{Uniform}(\Gamma^\new - 0.5, \Gamma^\new + 0.5) + U_i
\end{align}

We take $\mathcal{H}$ to be the reproducing kernel Hilbert space (RKHS) induced by the Gaussian kernel with bandwidth $\nu >0$.

In all the experiments below, we set $\sigma_\Gamma = 2$, $\Gamma^\new = 1$, $\sigma_U = 1$.

\begin{figure}[!htpb]
    \centering
    \includegraphics[width=0.65\linewidth]{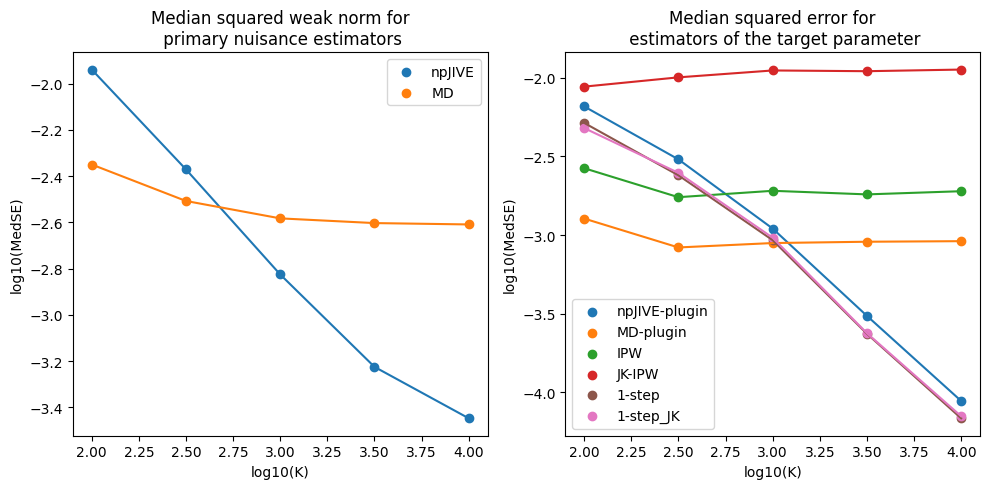}
    \caption{$n=30$, $\lambda_h=10^{-2}$, $\nu_h=1/3$, $L_h=5$, $\nu_\IPW = 1/10$, $L_\IPW = 10$, $w=3$ }
    \label{fig:enter-label}
\end{figure}

\begin{figure}
    \centering
    \includegraphics[width=0.65\linewidth]{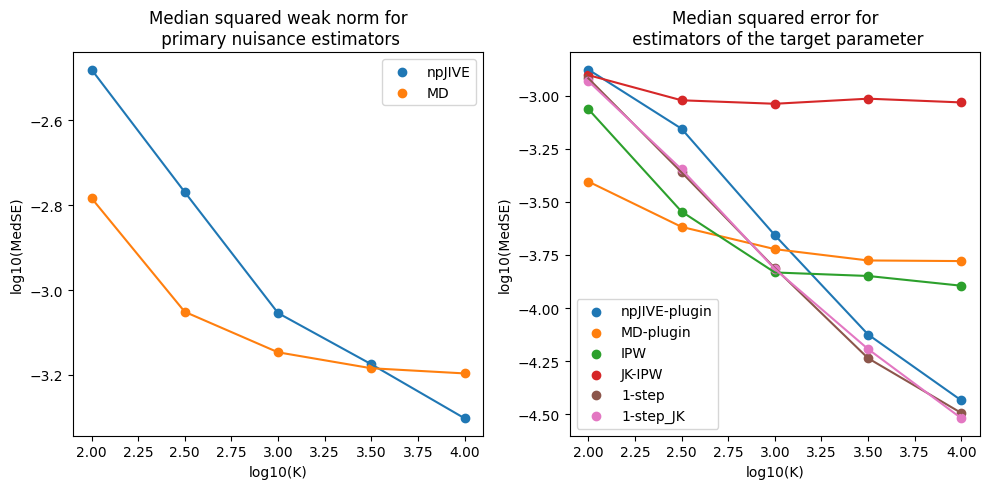}
    \caption{$n=100$, $\lambda_h=10^{-2}$, $\nu_h=1/4$, $L_h=7$, $\nu_\IPW = 1/10$, $L_\IPW = 10$, $w=3$ }
    \label{fig:enter-label}
\end{figure}

\begin{figure}
    \centering
    \includegraphics[width=0.65\linewidth]{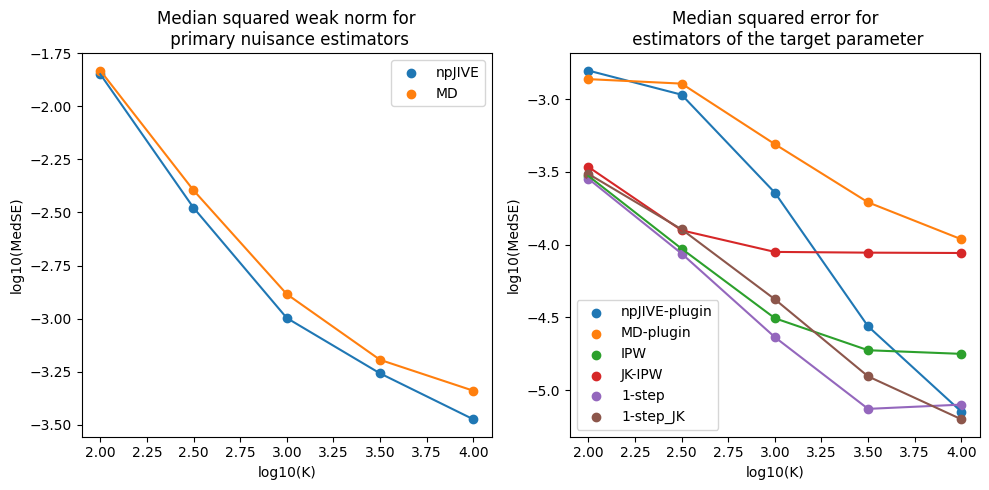}
    \caption{$n=300$, $\lambda_h=10^{-1}$, $\nu_h=1/10$, $L_h=10$, $\nu_\IPW = 1/10$, $L_\IPW = 10$, $w=3$ }
    \label{fig:enter-label}
\end{figure}

\begin{figure}
    \centering
    \includegraphics[width=0.65\linewidth]{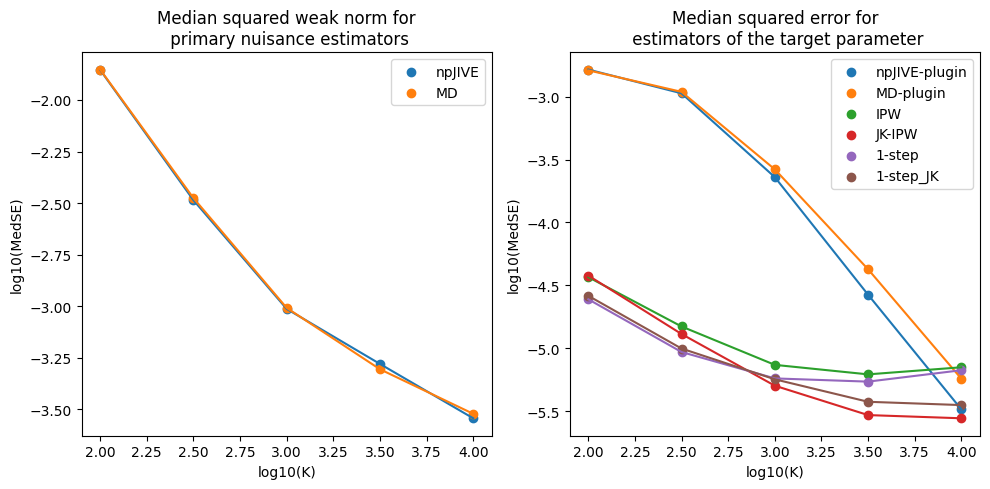}
    \caption{$n=3000$, $\lambda_h=10^{-1}$, $\nu_h=1/10$, $L_h=10$, $\nu_\IPW = 1/10$, $L_\IPW = 10$, $w=3$ }
    \label{fig:enter-label}
\end{figure}

As can be observed in the above figures, while they appear to be consistent as $n \to \infty$ and to significantly reduce median squared error when used to construct the one-step, the approximate-identification debiasing nuisances do not appear to be consistent as $K \to \infty$. This is in contrast with the the exact-identification debiasing nuisance which, under exact identification is actually consistent, as guaranteed by the results of section \ref{sec: estimation_debiasing_nuisance}. We demonstrate empirically the consistency of the exact identification debiasing nuisance in a numerical experiment in which we ensure exact identification. For each cell $a$, we draw a length-5 vector of fist-stage effects
\begin{align}
    \mu_a \sim \mathrm{Dirichlet}(10 \times (1,\ldots,1)).
\end{align}
Let $s_m = -1 + 2 m / 5$ for $m=1,\ldots,5$. Conditional on the first-stage effects, the data-generating process is as follows:
\begin{align}
U_i & \sim \mathrm{Unif}(-0.2, 0.2)\\
    S_i & \sim s_{\mathrm{Multinomial}(\mu_a)} + U_i\\
    Y_i  &= h(S_i) - 10 U_i.
\end{align}
We generate the data in the target novel experiment cell according to
\begin{align}
    U_i^\new & \sim \mathrm{Unif}(-0.2, 0.2)\\
    S_i^\new & \sim s_{\mathrm{Multinomial}(\mu^\new)} + U_i^\new,
\end{align}
where $\mu^\new = (0, 1/10, 2/10, 3/10, 4/10)$.
We report mean squared error, squared bias and variance as a function of $K$, under $n=100$ for this data-generating distribution in figure \ref{fig:exact_id}.

\begin{figure}
    \centering
    \includegraphics[width=0.75\linewidth]{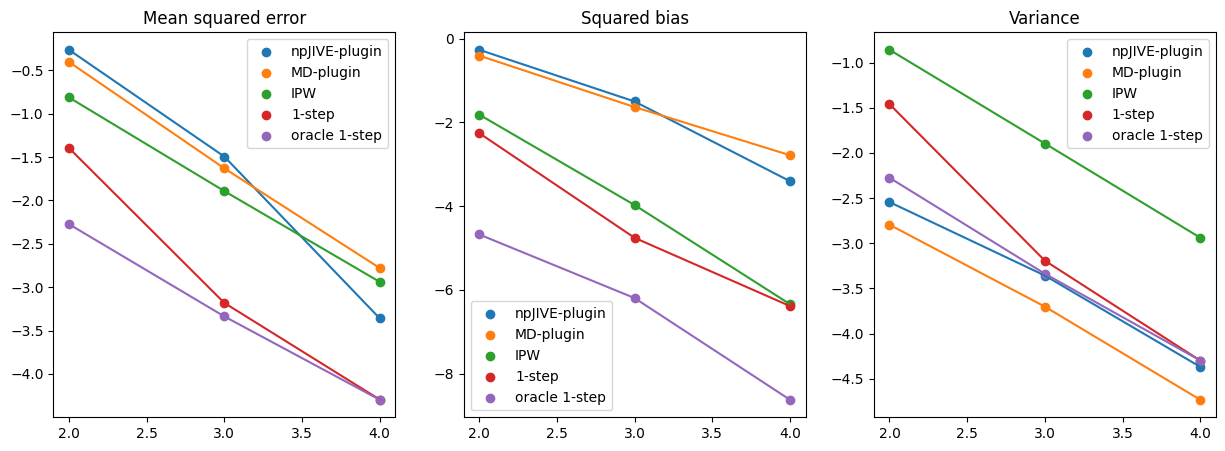}
    \caption{Exact identification.}
    \label{fig:exact_id}
\end{figure}

\pagebreak

\bibliographystyle{abbrvnat}
\bibliography{lit}

\appendix

\section{Proof of theorem \ref{thm:id_implies_strong_id}}

\begin{proof}
That $\overline{\mathcal{R}(\T_K^*)} = \mathcal{R}(\T_K^*)$ directly follows from the fact that $\mathcal{R}(\T_K^*)$ is a finite dimensional linear subspace, and therefore equals its closure.

For every $a \in [K]$, let $q_a = \bm{1}\{ \cdot = a\}$, and let $q_a^* = \T_K^* q_a$. Let $r$ be the rank of $q_1^*,\ldots, q_K^*$. Since $q_1,\ldots,q_K$ form a basis of $\mathcal{Q}$, $q_1^*,\ldots,q_K^*$ form a basis of $\mathcal{R}(\T_K^*)$ and therefore $\mathcal{R}(\T_K^*)$ has dimension $r$. Since $\mathcal{R}(\T_K^* \T_K) \subseteq \mathcal{R}(\T_K^*)$, the claim will follow if we can show that $\mathcal{R}(\T_K^* \T_K)$ has dimension at least $r$ too.

Without loss of generality, suppose that $q_1^*,\ldots,q_r^*$ are linearly independent. Let us show that $\T_K q_1^*, \ldots, \T_K q_r^*$ are linearly independent too. Observe that for any $a, a'$, 
\begin{align}
    \T_K q_a^*(a') = K \langle \T_Kq_a^*, q_a \rangle_{L_2(\mathcal{Q})} = K \langle q_a^*, q_{a'}^* \rangle_{L_2(\mathcal{S})}.
\end{align}
Therefore, for any $\beta_1,\ldots, \beta_r \in \mathbb{R}$ such that $\sum_{a=1}^r \beta_a \T_K q_a^* = 0$, we have 
\begin{align}
   0 =  \sum_{a'=1}^r \beta_{a'} \sum_{a=1}^r \beta_a \T_K q_a^*(a') = K \langle \sum_{a=1}^r \beta_a q_a^*,  \sum_{a=1}^r \beta_a q_a^* \rangle_{L_2(\mathcal{S})},
\end{align}
which implies that $\sum_{a=1}^r \beta_a q_a^* = 0$, which in turns implies that $\beta_1=\ldots=\beta_r = 0$ as $q_1^*,\ldots,q_r^*$ are linearly independent, therefore $\T_K q_1^*,\ldots,\T_K q_r^*$ are linearly independent. We have thus shown that application of $\T_K$ to $q_1^*,\ldots,q_r^*$ didn't reduce the rank. 

It remains to show that application of $\T_K^*$ to $\T_K q_1^*,\ldots,\T_Kq_r^*$ doesn't reduce the rank either. Suppose that $\beta_1,\ldots, \beta_r$ are such that $\sum_{a=1}^r \beta_a \T_K^* \T_K q_a^* = 0$. Taking inner products against $q_1^*, \ldots, q_r^*$, we have
\begin{align}
    0 = \sum_{a'=1}^r \beta_{a'} \langle \sum_{a=1}^r \beta_a  \T_K^* \T_K q_a^*, q_{a'}^* \rangle = \langle \sum_{a=1}^r \beta_a \T_K q_a^*, \sum_{a=1}^r \beta_a \T_K q_a^* \rangle,
\end{align}
which implies $\sum_{a=1}^r \beta_a \T_K q_a^* = 0$ and in turns implies $\beta_1=\ldots=\beta_r = 0$ since $\T_K q_1^*, \ldots \T_K q_r^*$ are linearly independent. Therefore, $\mathcal{R}(\T_K ^* \T_K)$ must have dimension at least $r$ which concludes the proof.
\end{proof}

\section{Primary nuisance estimation proofs}

For any two functions $f_1, f_2 : [K] \to \mathbb R$, let
\begin{align}
    \langle f_1, f_2 \rangle = \frac{1}{K} \sum_{a=1}^K f_1(a) f_2(a).
\end{align}

\subsection{Local maximal inequality}

Let $O_1,\ldots,O_N \in \mathcal{O}$ be $N$ independent random variables. For any function $f:\mathcal{O} \to \mathbb{R}$, let 
\begin{align}
    \|f\| = \sqrt{\frac{1}{N} \sum_{i=1}^N \EE[f(O_i)^2]}.
\end{align}

\begin{lemma}[Local maximal inequality]\label{lemma:loc_max_ineq}
Let $\mathcal{F}$ be a star-shaped class of functions such that $\sup_{f \in \mathcal{F}} \|f\|_\infty \leq 1$ and $\|f\| \leq r$. Suppose that $ N^{-1/2} \sqrt{\log \log \delta_N} =  o(\delta_N)$. Let $b_1,\ldots,b_N$ be fixed real numbers such that $\max_{i=1,\ldots,N} |b_i| \leq B$. Then, for any $u > 0$, it holds with probability $1 - e^{-u^2}$ that for every $f \in \mathcal{F}$
\begin{align}
    \frac{1}{N} \sum_{i=1}^N b_i(f(O_i) - \EE[f(O_i)]) \lesssim B \left(\delta_N \| f\| + \delta_N^2 + \frac{u \| f \|}{\sqrt{N}} + \frac{u^2}{N}\right).
\end{align}
\end{lemma}

\begin{proof}
    Let $\mathcal{S}_m = \{ f : \| f \| \leq r_m \}$ where $r_m = 2^m \delta_N$. From Talagrand's or alternatively Bousquet's inequality, it holds with probability $1-e^{-u^2}$ that for every $f \in \mathcal{S}_m$,
    \begin{align}
        \frac{1}{N} \sum_{i=1}^N b_i(f(O_i) - \EE[f(O_i)]) \lesssim \EE \left[ \sup_{f \in \mathcal{S}_m} \frac{1}{N} \sum_{i=1}^N \epsilon_i b_i f(O_i) \right] + B(\frac{u r_m}{\sqrt{N}} + \frac{u^2}{N}).
    \end{align}
    From the contraction lemma for Rademacher processes, 
    \begin{align}
       \EE \left[ \sup_{f \in \mathcal{S}_m} \frac{1}{N} \sum_{i=1}^N \epsilon_i b_i f(O_i) \right] \lesssim B \mathcal{R}_N(\mathcal{F}, r_m).
    \end{align}
    Let $m(f) = \min \{m \geq 0 : f \in \mathcal{S}_m\}$. From a union bound, with probability $1-e^{-u^2}$, for every $f \in \mathcal{F}$,
    \begin{align}
        \frac{1}{N} \sum_{i=1}^N b_i(f(O_i) - \EE[f(O_i)]) \lesssim B\left(\mathcal{R}_N(\mathcal{F}, r_{m(f)}) + (u + \sqrt{\log \log_2 \delta_N})\frac{r_{m(f)}}{\sqrt{N}} + \frac{u^2 \log \log_2 \delta_N}{N}\right).
    \end{align}
    If $m(f) = 0$, $\mathcal{R}_N(\mathcal{F}, r_{m(f)}) = \mathcal{R}_N(\mathcal{F}, \delta_N) \leq \delta_N^2$ by definition of the critical radius. If $m(f) \geq 1$, 
    \begin{align}
        \frac{\mathcal{R}_N(\mathcal{F}, r_{m(f)} )}{ r_{m(f)} } \leq \delta_N,
    \end{align}
    therefore $\mathcal{R}_N(\mathcal{F}, r_{m(f)} ) \leq \delta_N \| f\|$. Therefore, with probability $1-e^{-u^2}$, it holds that for every $f \in \mathcal{F}$,
    \begin{align}
         \frac{1}{N} \sum_{i=1}^N b_i(f(O_i) - \EE[f(O_i)]) \lesssim B \left(\delta_N \| f\| + \delta_N^2 + \frac{u \|f\|}{\sqrt{N}} + \frac{u^2}{N} \right).
    \end{align}
\end{proof}

\subsection{Decomposition of the difference between population and empirical risk}

We have that
\begin{align}
    &2(\widehat R_{1,0}(h) - R_1(h) ) \\
    =&\langle \widehat \T_{K,0} (h_0 - h + \eta), \widehat \T_{K,1} (h_0 - h + \eta) \rangle - \langle \T_K (h_0 - h), \T_K (h_0 - h) \rangle \\
    =& \langle  (\widehat \T_{K,0} - \T_K)(h-h_0), (\widehat \T_{K,1} - \T_K)(h-h_0) \rangle  \label{eq:h_U_process}\\
    &+ \sum_{v=0,1} \langle \T_K (h - h_0),  (\widehat\T_{K,v} - \T_K) (h - h_0) \rangle \label{eq:h_EP} \\ 
    &+ \sum_{v=0,1} \langle \T_K ( h - h_0),  (\widehat\T_{K,v} - \T_K) \eta \rangle \label{eq:h_err_EP} \\
    &+ \sum_{v=0,1} \langle ( \widehat \T_{K,1-v} - \T_K) ( h - h_0),  (\widehat\T_{K,v} - \T_K) \eta \rangle \label{eq:h_err_U_process} \\ 
    & + \langle (\widehat \T_{K,0} - \T_K) \eta, (\widehat \T_{K,1} - \T_K) \eta \rangle \label{eq:h_quad_eps_term}
\end{align}

\subsection{U-process terms}

Denote $\mathcal{G} = \mathcal{H} - h_0$. For any $g \in \mathcal{G}$, let $g^0 = g - \T_K g$, and let $\mathcal{G}^0 = \{ g - \T_K g : g \in \mathcal{G} \}$ For any such $g^0$, let 
\begin{align}
    Z_{N,1}^0(g) =  \langle \widehat \T_{K,0}  g^0, \widehat \T_{K,1} g^0 \rangle.
\end{align}
For any $a \in [K]$, $\delta > 0$ let
\begin{align}
    \mathcal{R}_{n,a}(\mathcal{G}^0, \delta) = \EE \left[ \sup_{g^0 \in \mathcal{G}^0 : \|g^0 \|_{2,a} \leq \delta} \frac{2}{n} \sum_{i: A_i = a, V_i = 0} \epsilon_i g^0(S_i) \right]. 
\end{align}

\begin{lemma}
    Suppose that the classes $\mathcal{G}^0$ and $(\widehat \T_{K,0} \mathcal{G}^0) \times \mathcal{G}^0$ are star-shaped. Let $\widetilde \delta_n$ be any solution to $\mathcal{R}_{n,a}(\mathcal{G}^0, \delta) \leq \delta^2$ for every $a \in [K]$, let $\widehat \delta_N$ be the $\| \cdot \|$ critical radius of $(\widehat \T_{K,0} \mathcal{G}^0) \times \mathcal{G}^0$ conditional on $\mathcal{D}_0$, and let $\widehat{r}_N^2$ be an upper bound on the Rademacher complexity of $(\widehat \T_{K,0} \mathcal{G}^0) \times \mathcal{G}^0$ conditional on $\mathcal{D}_0$. Let $\zeta > 0$. Suppose that $\sqrt{\log(K/\zeta) / n} = o(\widetilde \delta_n)$, $\sqrt{\log \log \widetilde \delta_n / n} = o(\widetilde \delta_n)$, $\widehat \delta_N = o(\sqrt{\log \log \widehat \delta_N / N})$. Then, with probability $1 - 3 \zeta$, it holds for every $g^0 \in \mathcal{G}^0$ that
    \begin{align}
        Z_N^0(g) = O\left( \left( \widehat \delta_N \widetilde \delta_n  \|g^0\|^2 +  \widehat \delta_N \widetilde \delta_n^2  \|g^0\| +\widehat \delta_N^2) \right) \wedge \widehat{r}_N^2 \right),
    \end{align}
    where the constant in $O(\cdot)$ is uniform over $\mathcal{G}^0$.
\end{lemma}

\begin{proof}
    From lemma \ref{lemma:loc_max_ineq}, a union bound and the condition $\sqrt{\log (\zeta / K) / n} = o(\widetilde \delta_n)$, it holds with probability $1 - \zeta$ that, for every $a \in [K]$,
    \begin{align}
        [\widehat \T_{K,0} g^0](a) = O(\widetilde \delta_n \|g^0\|_{2,a} + \widetilde \delta_n^2).
    \end{align}
    Let $\mathcal{E}$ be the event that the above display holds. Conditional on $\mathcal{D}_0$, under thes assumption that  $\sqrt{\log \log \widetilde \delta_n / n} = o(\widetilde \delta_n)$ and $\widehat \delta_N = o(\sqrt{\log \log \widehat \delta_N / N})$, from lemma \ref{lemma:loc_max_ineq}, it holds that with probability $1 - \zeta$ that, for any $g^0 \in \mathcal{G}^0$,
    \begin{align}
        Z_{N,1}(g^0) \leq \widehat \delta_N \| \widehat \T_{K,0} g^0 \times g \|_{|\mathcal{D}_0} + \widehat \delta_N^2.
    \end{align}
    where $\| \cdot \|_{|\mathcal{D}_0}$ is the $\| \cdot \|$ norm conditional on $\mathcal{D}_0$.

    Under the event $\mathcal{E}$, for any $g^0 \in \mathcal{G}^0$, 
    \begin{align}
        \| \widehat \T_{K,0} g^0 \times g \|_{|\mathcal{D}_0} = &\sqrt{\frac{2}{N} \sum_{i : V_i = 1}  \left\lvert \widehat \T_{K,0} g^0 (A_i)\right\rvert^2 \EE[ (g^0(S_i))^2] } \\
        \leq & \sqrt{  \frac{2}{N} \sum_{i : V_i = 1} (\widetilde \delta_n \| g^0 \|_{2,A_i} + \widetilde \delta_n^2)^2 \| g^0 \|_{2,A_i}^2 } \\
        \lesssim  & \widetilde \delta_n \sqrt{ \frac{2}{N} \sum_{i : V_i = 1} \| g^0 \|_{2,A_i}^4 } + \widetilde \delta_n^2  \|g^0\| \\
        \lesssim & \widetilde \delta_n  \|g^0\|^2 + \widetilde \delta_n^2  \|g^0\|.
    \end{align}
    A $(1-\zeta)$-probability $O(\widehat{r}_N^2)$ bound directly follows from applying Talagrand's or Bousquet's inequality conditionally on $\mathcal{D}_0$. The claim then follows from a union bound.
\end{proof}

Now, let 
\begin{align}
    Z_{N,2}^0(g) = \langle \widehat \T_{K,0} \eta, \widehat \T_{K,1} g^0 \rangle
\end{align}

\begin{lemma}
    Let $\widetilde \delta_N$ be an upper bound on the critical radius of $\mathcal{G}^0$. Let $\zeta \in (0,1)$. Suppose that $\log (1 / \zeta) = o(\log K)$, $\sqrt{\log (1 / \zeta)/N} = o(\widetilde \delta_N)$ and $\sqrt{\log \log \widetilde \delta_N / N} = o(\widetilde \delta_N)$. Then, with probability $1 - O(\zeta)$, for every $g \in \mathcal{G}$,
    \begin{align}
        Z_{N,2}^0(g) = O \left( \sqrt{\frac{\log K}{n}}(\widetilde \delta_N \| g^0 \| + \widetilde \delta_N^2) \right).
    \end{align}
\end{lemma}

\begin{proof}
    From Hoeffding, a union bound, and that $\log (1 / \zeta) = o(\log K)$, it holds with probability $1-\zeta$ that 
    \begin{align}
        \max_{a \in [K]} |[\widehat\T_{K,0} \eta](a)| = O \left( \sqrt{\frac{\log K}{n}} \right).
    \end{align}
    Observe that 
    \begin{align}
         Z_{N,2}^0(g) = \frac{1}{N/2} \sum_{i : V_i = 1} [\widehat\T_{K,0} \eta](A_i) g^0(S_i).
    \end{align}
    The conclusion then follows by applying lemma \ref{lemma:loc_max_ineq} conditional on $\mathcal{D}_0$, on the event that the above display holds, with $b_i=[\widehat\T_{K,0} \eta](A_i)|$.
\end{proof}

\subsection{Empirical process terms}

\subsubsection{First empirical process term}
We have that 
\begin{align}
    &\langle \T_K (h - h_0), (\widehat \T_{K,1} - \T_K)(h-h_0) \rangle \\
    &= \frac{2}{N} \sum_{i: V_i = 1} [\T_K(h-h_0)](A_i) (h-h_0)(S_i) - E[[\T_K(h-h_0)](A_i) (h-h_0)(S_i)],
\end{align}
Therefore, from a direct application of lemma \ref{lemma:loc_max_ineq}, we have that, as long as $\sqrt{\log(1/ \zeta) / N} = o(\bar \delta_N)$ and $\sqrt{\log \log (\delta_N) / N} = o(\delta_N)$, it holds with probability $1-2\zeta$ that for every $h \in \mathcal{H}$,
\begin{align}
    &\langle \T_K (h - h_0), (\widehat \T_{K,1} - \T_K)(h-h_0) \rangle \\
    &\leq \left\lVert [\T_K ( h - h_0)] \times (h - h_0)  \right\rVert \bar \delta_N + \bar \delta_N^2 \\
    & \leq \left\lVert [\T_K ( h - h_0)] \right\rVert \bar \delta_N + \bar \delta_N^2.
\end{align}

\subsubsection{Second empirical process term}

Let $\breve \delta_N$ be an upper bound on the critical radius of $\T_K \mathcal{G}$.

We have that 
\begin{align}
    \langle \T_K g, \widehat \T_{K,0} \eta \rangle = & \frac{1}{N/2} \sum_{i : V_i = 0} [\T_K g](A_i) \eta_i - \EE[ [\T_K g](A_i) \eta_i ],
\end{align}
Since the $\eta_i$'s are 1-subGaussian, it holds with probability $1- \zeta$ that 
\begin{align}
    \max_{i \in [N]} | \eta_i | = O(\sqrt{\log (N / \zeta)}).
\end{align}
Therefore, from lemma \ref{lemma:loc_max_ineq} applied conditional on $\mathcal{D}_0$, on the event that the above display holds, it holds with probability $1- O(\zeta)$ that, provided the lemma's conditions on $\zeta$ and $\log \log \breve{\delta}_N$ hold,
\begin{align}
    \langle \T_K g, \widehat \T_{K,0} \eta \rangle = O \left( \sqrt{\log N} \left( \breve \delta_N \| \T_K g \| + \breve \delta_N^2 \right) \right).
\end{align}

\subsection{Empirical and population risk difference for the primary nuisance}

\begin{table}[!h]
    \centering
    \renewcommand{\arraystretch}{1.5}
    \begin{tabular}{c|c|c|c}
        Term & CR notation & CR classes & Rate (up to $\log$ terms) \\
        \hline
            $\langle  (\widehat \T_{K,0} - \T_K)g, (\widehat \T_{K,1} - \T_K)g \rangle$ 
            & $\widetilde \delta_n$, $\widehat \delta_N$ & $\mathcal{G}^0$, $(\widehat \T_{K,0} \mathcal{G}^0) \times \mathcal{G}^0$
            & $\widehat \delta_N \widetilde \delta_n \| g^0 \|^2 + \widehat \delta_N \widetilde \delta_n^2 \| g^0 \| + \widehat \delta_N^2$
            \\
        \hline
             $\langle \T_K g,  (\widehat\T_{K,v} - \T_K) g \rangle$
             & $\bar \delta_N$ 
             & $(\T_K \mathcal{G}) \times \mathcal{G}$
             & $\bar \delta_N \| \T_K g \| + \bar \delta_N^2$ \\
        \hline
            $\langle \T_K g, \widehat \T_{K,v} \eta \rangle$ 
            & $\breve \delta_N$ 
            & $\T_K \mathcal{G}$
            & $\breve \delta_N \| \T_K g \| + \breve \delta_N^2$ \\
        \hline
        $\langle ( \widehat \T_{K,1-v} - \T_K) g,  (\widehat\T_{K,v} - \T_K) \eta \rangle$
        & $\widetilde \delta_N$
        & $\mathcal{G}^0$
        & $n^{-1/2}(\widetilde \delta_N \| g^0 \| + \widetilde \delta_N^2)$
    \end{tabular}
    \caption{Summary of localized rates for the terms in the equicontinuity difference}
    \label{tab:my_label}
\end{table}
\begin{lemma}\label{lemma:emp_pop_risk_diff_primary}
    Suppose the assumptions of the above lemmas hold. Suppose that $\widetilde \delta_n^2 = O(n^{-1/2})$ and that $n^{-1/2} = O(\widetilde \delta_n)$. Let $\delta_N = \max(\widetilde \delta_N, \breve \delta_N, \widehat \delta_N)$. Then it holds with probability at least $1-O(\zeta)$ that for every $h$ in $\mathcal{H}$,
    \begin{align}
         \widehat R_{1,0}(h) - R_1(h) = O\left(  \delta_N \| \T_K (h - h_0) \| + \delta_N \widetilde \delta_n \| h - h_0 \|^2 + \delta_N n^{-1/2} \| h - h_0 \| + \delta_N^2 \right).
    \end{align}
\end{lemma}

\subsection{Weak norm convergence rate for the primary nuisance}

\begin{lemma}[Convergence primary nuisance]\label{lemma:conv_primary_nuisance}
    Suppose the assumptions of \ref{lemma:emp_pop_risk_diff_primary} hold, and that $\delta_N \widetilde \delta_n = O(\lambda)$. Then it holds with probability $1 - O(\zeta)$ that 
    \begin{align}
        \| \T_K (\widehat h - h_0)\|^2 = O(\lambda + \delta_N \delta_n).
    \end{align}
\end{lemma}

\begin{proof}
All the statements in this proof hold with probability $1 - c e^{-\delta}$. From lemma \ref{lemma:emp_pop_risk_diff_primary}, optimality of $\widehat h$ for the empirical risk, and a another application of lemma \ref{lemma:emp_pop_risk_diff_primary},
\begin{align}
    &\| \T_K ( \widehat h - h_0 ) \|^2 \\
    \leq &\| \T_K( h^* - h_0) \|^2 \\
    &+ O\left( \delta_N ( \|\T_K ( \widehat h - h^*) \| + \| \T_K ( h^* - h_0)\| ) \right.\\
    & \qquad + \delta_N \delta_n ( \| \widehat h - h^*\|^2 + \| h^* - h_0\|^2 ) \\
    & \qquad + \delta_N n^{-1/2} ( \| \widehat h - h^*\| + \| h^* - h_0\| ) \\
    &  \qquad + \left.\delta_N^2 \right) \\
    &+ \lambda (\|h^*\|_{2,N}^2 - \| \widehat h \|_{2,N}^2) \\
    \leq & 2 \| \T_K ( h_0 - h^* ) \|^2 \\
    &+ O( \delta_N \|\T_K ( \widehat h - h^*) \| + \delta_N \delta_n \| \widehat h - h^* \|^2 + \delta_N n^{-1/2} \| \widehat h - h^* \| +\delta_N \delta_n) \\
    &+ \lambda (\|h^*\|_{2,N}^2 - \| \widehat h \|_{2,N}^2).
\end{align}
From an analysis of the second-order polynomial $t \mapsto \| \T_K ( h^* + t (\widehat h - h^*) - h_0) \|^2 + \|h^* + t (\widehat h - h^*) - h_0 \|^2$, 
\begin{align}
    \| \T_K( \widehat h - h^* ) \|^2 + \lambda \| \widehat h - h^*  \|^2 \leq \| \T_K (\widehat h - h_0 )\|^2 - \| \T_K (h^* - h_0 )\|^2 - \lambda ( \| h^*\|^2 - \|\widehat h \|^2).
\end{align}
Combining this with the preceding display and the fact that $\| h^*\|^2 - \|\widehat h \|^2 - \| h^*\|^2_{2,N} - \|\widehat h \|^2_{2,N} = O(\| \widehat h - h^* \| \delta_N + \delta_N^2 )$ from lemma \ref{lemma:loc_max_ineq} yields that 
\begin{align}
     &\| \T_K( \widehat h - h^* ) \|^2 + \lambda \| \widehat h - h^*  \|^2 
     \\ \leq & \| \T_K( h^* - h_0 ) \|^2 \\
      & + O (  \delta_N \| \T_K (h^* - h_0) \| + \delta_N \delta_n \| \widehat h - h^* \|^2 + \delta_N \delta_n^2 \| \widehat h - h^* \| + \delta_N \delta_n + \lambda \delta_N \| \widehat h - h^* \| ).
\end{align}
Applying the AM-GM inequality to $ \delta_N \| \T_K (h^* - h_0) \|$ and using that $\delta_N \delta_n = O(\lambda)$ and that $\| \T_K( h^* - h_0 ) \|^2 = O(\lambda)$ yields that
\begin{align}
    \frac{1}{2} \| \T_K (\widehat h - h^*) \|^2 + \frac{\lambda}{2} \| \widehat h - h^*\|^2 = O(\lambda +  \delta_N \delta_n),
\end{align}
which yields the claim.
\end{proof}

\subsection{Consistency in strong norm of the primary nuisance estimator}

We have that 
\begin{align}
    \frac{1}{2} \| \T_K (\widehat h - h^\dagger ) \|^2 \leq & R_1(\widehat h) - R_1(h^\dagger) \\
    \leq & \widehat R_{1,0}(\widehat h) - \widehat R_1(h^\dagger) + O(\delta_N \| \T_K(\widehat h - h^\dagger) \|+ \delta_N \widetilde \delta_n ) \\
    \leq &  O(\delta_N \| \T_K(\widehat h - h^\dagger) \| + \delta_N \widetilde \delta_n ) - \lambda ( \|\widehat h \|_{2,N}^2 - \| h^\dagger \|^2_{2,N} ).
\end{align}
Therefore, from AM-GM,
\begin{align}
    \| \widehat h \|_{2,N}^2 - \| h^\dagger \|^2_{2,N} = O\left(\frac{\delta_N \widetilde \delta_n}{\lambda}\right).
\end{align}
Therefore, from the definition of $h^\dagger$, which implies the limit of the LHS is lower bounded by zero, we have that $\lim \|\widehat h \| = \| h^\dagger\|$, which implies the conclusion.

\section{Convergence of the debiasing nuisance}

\subsection{Regularization bias under source}

\begin{lemma}\label{lemma:reg_bias_source}
    Under the source condition for $\alpha$ with parameter $\beta$, we have that
    \begin{align}
        \| \xi - \xi_0 \| = O ( \| w_0 \|^2 \lambda^{\beta}) \qquad \text{and}  \qquad \| \T_K (\xi - \xi_0) \|^2 = O ( \| w_0 \|^2 \lambda^{1+\beta}).
    \end{align}
\end{lemma}

The proof is identical to that of lemma 5 in \cite{bennett2023source}.

\subsection{Decomposition of the equicontinuity term}

The equicontinuity term decomposes as the sum of quadratic term and of a linear term as follows:
\begin{align}
    &(R(\xi) - R(\xi_0)) - (\widehat R(\xi) - \widehat R(\xi_0))\\
    =& Q(\xi, \xi_0) + L(\xi,\xi_0),
\end{align}
where
\begin{align}
    &Q(\xi, \xi_0) \\
    =& \frac{1}{2} \langle (\widehat \T_{K,0} - \T_K) (\xi - \xi_0), (\widehat \T_{K,1} - \T_K) (\xi - \xi_0)\rangle \\
    &+ \sum_{v=0,1} \left\lbrace \langle \T_K (\xi - \xi_0), (\widehat \T_{K,v} - \Pi) (\xi - \xi_0)\rangle \right. \\
    &\qquad +  \langle (\widehat \T_{K,v} - \Pi) \xi_0, (\widehat \T_{K,1-v} - \Pi) (\xi - \xi_0)\rangle \\
    &\left. \qquad + \langle \Pi \xi_0, (\widehat \T_{K,v} - \Pi) (\xi - \xi_0)\rangle \right\rbrace.
\end{align}
and 
\begin{align}
    L(\xi, \xi_0) = E[\xi(S') - \xi_0(S')].
\end{align}

\subsection{Bounding the equicontinuity term}

Following the same arguments as for the primary nuisance, 
we obtain that
\begin{align}
    Q(\xi, \xi_0) =  \widetilde O  & \left(  \widehat \delta_N \delta_n \|\xi - \xi_0\|^2 + \widehat \delta_N \widetilde \delta_n^2 \| \xi - \xi_0 \| + \widehat \delta_N^2 \right.\\
    & + \bar \delta_N \| \Pi (\xi - \xi_0)\| + \bar \delta_N^2 \\
    & + n^{-1/2} \widetilde \delta_N \| \xi - \xi_0 \| +  \widetilde \delta_N^2\\
    & + \left. \widetilde \delta_N \| \xi - \xi_0 \| + \widetilde \delta_N^2 \right) \\
    = & \widetilde O \left( \delta_N \| \xi - \xi_0 \|  + \delta_N \widetilde \delta_n \| \xi - \xi_0\|^2 + \delta_N ^2 \right).
\end{align}

From lemma \ref{lemma:loc_max_ineq}, 
\begin{align}
    L(\xi, \xi_0) = \widetilde O \left( \delta'_{n'} \| \xi - \xi_0 \| + (\delta_{n'}')^2 \right).
\end{align}

Therefore, since $n' = o(N)$, we have that
\begin{align}
    (R(\xi) - R(\xi_0)) - (\widehat R(\xi) - \widehat R(\xi_0)) = \widetilde O \left( \delta'_{n'} \| \xi - \xi_0 \| + (\delta_{n'}')^2 \right).
\end{align}

\subsection{Weak norm bound}

From an analysis of $t \mapsto \| \Pi (t (\widehat \xi - \xi^*) + (\xi^* - \xi_0)) \|^2 + \lambda \| (t (\widehat \xi - \xi^*) + \xi^*) \|^2$, it holds that
\begin{align}
    \| \Pi (\widehat \xi - \xi^*) \|^2  + \lambda \| \widehat \xi - \xi^* \| \leq & \| \Pi (\widehat \xi - \xi_0) \|^2 - \| \Pi (\xi^* - \xi_0) \|^2 + \lambda \{ \| \widehat \xi  \|^2 - \| \xi^* \|^2 \}.
\end{align}
From the equicontinuity term bound applied twice and optimality of $\widehat \xi$ for the penalized empirical risk, we have

\begin{align}
    \| \Pi (\widehat \xi - \xi^*) \|^2  + \lambda \| \widehat \xi - \xi^* \|^2 \leq & \| \Pi ( \xi^* - \xi_0) \|^2 + \widetilde  O \left( \delta_{n'}' \| \widehat \xi - \xi_0 \| + (\delta_{n'}^{'})^2 \right) \\
    & + \lambda \{ (\| \widehat \xi  \|^2  - \| \widehat \xi \|_{2,N}^2) - (\| \xi^* \|^2 - \| \xi^* \|_{2,N}^2) \} \\
    \leq & \lambda^{1 + \beta} + \widetilde O\left( \delta'_{n'} \| \widehat \xi  - \xi^* \| + \delta_{n'}' \|\xi^* - \xi_0 \| + (\delta_{n'}')^2 \right) \\
    \leq & \lambda^{1 + \beta} + \widetilde O\left( \delta'_{n'} \| \widehat \xi  - \xi^* \| + \delta_{n'}' \lambda^{\beta / 2} + (\delta_{n'}')^2 \right), & \label{eq:xi_strong_convx_eqct_and_opt_bound}
\end{align}

where we have used in the second inequality that $\| \Pi ( \xi^* - \xi_0) \|^2 \leq \lambda^{1 + \beta}$ from lemma \ref{lemma:reg_bias_source}, and that $(\| \widehat \xi  \|^2  - \| \widehat \xi \|_{2,N}^2) - (\| \xi^* \|^2 - \| \xi^* \|_{2,N}^2) = O( \delta_N  \| \widehat \xi - \xi^* \| + \delta_N^2 )$ (see proof of theorem 4 in \cite{bennett2023source}), and in the third inequality that $\| \xi^* - \xi_0 \|^2 \leq \lambda^\beta$, from lemma \ref{lemma:reg_bias_source}.

Therefore, 
\begin{align}
    \| \Pi (\widehat \xi - \xi^*) \|^2 =& \widetilde O\left( \sup_{x \in \mathbb{R}} - \lambda x^2 + \delta_{n'}' x + \lambda^{1+\beta} + \delta_{n'}' \lambda^{\beta/2} + (\delta_{n'}')^2 \right) \\
    = &\widetilde O\left(\lambda^{-1} (\delta_{n'}')^2  + \lambda^{1+\beta} + \delta_{n'}' \lambda^{\beta/2}  \right),
\end{align}
and then, for $\lambda = (\delta_{n'}')^{\frac{2}{2+\beta}}$, we have
\begin{align}
    \| \Pi (\widehat \xi - \xi^*) \|  = \widetilde O\left(  (\delta_{n'}')^{\frac{1+\beta}{2+\beta}} \right),
\end{align}
and since from lemma \ref{lemma:reg_bias_source}, we have $\| Pi (\xi^* - \xi_0) \|  = O ( \lambda^{\frac{1+ \beta}{2}}) = O((\delta_{n'}')^{\frac{1+\beta}{2+\beta}})$, the triangle inequality implies the claim.

\subsection{Strong norm bound}

From \eqref{eq:xi_strong_convx_eqct_and_opt_bound}, 
\begin{align}
    \| \widehat \xi - \xi^* \|^2 = \widetilde O \left( \lambda^\beta + \frac{\delta_{n'}'}{\lambda} \| \widehat \xi - \xi^* \| + \delta_{n'}' \lambda^{\frac{\beta}{2} - 1} \right).
\end{align}
From the AM-GM inequality,
\begin{align}
    \| \widehat \xi - \xi^* \|^2 = \widetilde O \left( \lambda^\beta + \lambda^{-2} (\delta_{n'}')^2 + \delta_{n'}' \lambda^{ \frac{\beta}{2} - 1} \right).
\end{align}
For $\lambda = (\delta_{n'}')^{\frac{2}{2+\beta}}$, we then have that
\begin{align}
    \| \widehat \xi - \xi^* \|^2 = \widetilde O\left( (\delta_{n'}')^{\frac{2 \beta}{2+\beta}} \right),
\end{align}
and since $\| \xi^* - \xi_0 \|^2 = O(\lambda^\beta) = O((\delta_{n'}')^{\frac{2 \beta}{2+\beta}})$, the triangle inequality implies the claim.
\newpage

\end{document}